\documentclass[12pt,oneside]{amsart}
\usepackage{geometry}
\usepackage{enumitem,verbatim}
\usepackage{microtype}
\usepackage{physics}

\usepackage[pdfusetitle]{hyperref}
\usepackage{color}
\hypersetup{colorlinks,linkcolor=blue,citecolor=cyan}
\usepackage{amsthm}
\usepackage[nameinlink]{cleveref}
\usepackage{amsmath,amssymb}

\newtheorem{cor}{Corollary}
\newtheorem{lem}{Lemma}
\newtheorem{prop}{Proposition}
\newtheorem{thm}{Theorem}
\theoremstyle{definition}
\newtheorem{defn}{Definition}
\newtheorem{eg}{Example}

\newtheorem{qn}{Question}

\newcommand{\ang}[1]{\langle #1\rangle}
\newcommand{\caleft}{\curvearrowleft}
\newcommand{\car}{\curvearrowright}
\newcommand{\fin}{\mathrm{fin}}

\newcommand{\mb}{\mathbb}
\newcommand{\N}{\mb N}
\newcommand{\ol}{\overline}
\newcommand{\Q}{\mb Q}
\newcommand{\R}{\mb R}
\newcommand{\tl}{\triangleleft}
\newcommand{\Z}{\mb Z}
\DeclareMathOperator{\Aut}{Aut}
\DeclareMathOperator{\Inn}{Inn}

\DeclareMathOperator{\Out}{Out}
\DeclareMathOperator{\SO}{SO}

\title{Quotients by countable subgroups are hyperfinite}
\date{\today}

\author{Joshua Frisch}
\address{Division of Physics, Mathematics and Astronomy,
California Institute of Technology,
1200 E. California Blvd.,
Pasadena,
CA 91125}
\email{jfrisch@caltech.edu}

\author{Forte Shinko}
\address{Division of Physics, Mathematics and Astronomy,
California Institute of Technology,
1200 E. California Blvd.,
Pasadena,
CA 91125}
\email{fshinko@caltech.edu}

\thanks{The authors were partially supported by NSF Grant DMS-1464475.}

\begin{document}
\maketitle

\begin{abstract}
  We show that for any Polish group $G$ and any countable normal subgroup $\Gamma\tl G$,
  the coset equivalence relation $G/\Gamma$ is a hyperfinite Borel equivalence relation.
  In particular,
  the outer automorphism group of any countable group is hyperfinite.
\end{abstract}

\section{Introduction}
The purpose of this paper is to study the complexity of quotient groups $G/\Gamma$ from the point of view of descriptive set theory.
In particular,
we focus on the case where $G$ is Polish and $\Gamma$ is a countable normal subgroup.
If $\Gamma$ is a countable group,
then the automorphism group of $\Gamma$ has a natural Polish group structure,
and thus the outer automorphism group of $\Gamma$ is an example,
as is any countable subgroup of an abelian group. 

A major recent program is the study of complexity of ``definable'' equivalence relations.
Results in this area are often interpreted to be statements about the difficulty of classification of various natural mathematical objects. 
A particular focus of the theory of definable equivalence relations,
and one where much progress has recently been made,
is the study of Borel equivalence relations for which every class is countable,
the so-called countable Borel equivalence relations.
There is a natural preorder on Borel equivalence relations,
called Borel reduction,
where $E$ reducing to $F$ is interpreted as $E$ being ``easier'' than $F$.
The theory of countable Borel equivalence relations has been applied in numerous areas of mathematics.
For example,
the classification of finitely generated groups \cite{TV99},
of subshifts \cite{Cle09},
and the arithmetic equivalence of subsets of $\N$ \cite{MSS16} are all equally difficult.
In fact,
they are equivalent to the universal countable Borel equivalence relation $E_\infty$,
which is the hardest countable Borel equivalence relation.
On the other hand,
many other classification problems are easier.
For example,
classification of torsion-free finite rank abelian groups is substantially below $E_\infty$ \cite{Tho03, Tho09}.

Countable Borel equivalence relations can be characterized as those equivalence relations arising
from continuous actions of countable groups on Polish spaces,
and thus have very strong interplay with dynamics and group theory.
By a foundational result of Slaman-Steel and Weiss \cite{SS88, Wei84},
the equivalence relations which arise from a continuous (or more generally, Borel) action of $\Z$
are exactly the \textbf{hyperfinite} equivalence relations,
which are those which can be written as an increasing union of finite Borel equivalence relations.
More generally,
it has been shown that every Borel action of a countable abelian group \cite{GJ15},
and even of a countable locally nilpotent group \cite{SS13},
is hyperfinite. 
It is an open question whether this holds for all countable amenable groups.
By a theorem of Harrington-Kechris-Louveau \cite{HKL90},
the hyperfinite equivalence relations only occupy the first two levels of the hierarchy of
countable Borel equivalence relations on uncountable Polish spaces under Borel reduction,
and thus are considered to have low Borel complexity.

In general,
if $G$ is a Polish group and $\Gamma\le G$ is a countable subgroup,
then $G/\Gamma$ can be rather complicated;
we will give a non-hyperfinite example in \Cref{Preliminaries}.
However,
perhaps surprisingly,
if $\Gamma$ is a \textit{normal} subgroup of $G$,
then the coset equivalence relation $G/\Gamma$ must have low Borel complexity:
\begin{thm}
  Let $G$ be a Polish group and let $\Gamma$ be a countable normal subgroup of $G$.
  Then $G/\Gamma$ is hyperfinite.
\end{thm}
Notably,
in contrast to the aforementioned results,
we require no hypotheses on the algebraic structure of the acting group.
The proof proceeds by showing that the equivalence relation is generated by
a Borel action of a countable abelian group,
which is sufficient by the aforementioned theorem of Gao and Jackson.  

We obtain as a consequence the following result about outer automorphism groups:
\begin{cor}
  Let $\Gamma$ be a countable group.
  Then $\Out(\Gamma)$ is hyperfinite.
\end{cor}

If $G$ is a compact group with a countable normal subgroup $\Gamma\tl G$,
then we also show that the algebraic structure of $\Gamma$ is severely restricted:
\begin{thm}
  Let $G$ be a compact group and let $\Gamma$ be a countable normal subgroup of $G$.
  Then $\Gamma$ is locally virtually abelian,
  i.e.,
  every finitely generated subgroup of $\Gamma$ is virtually abelian.
\end{thm}

\subsection*{Acknowledgments}
We would like to thank Aristotelis Panagiotopoulos for posing the original question about outer automorphisms,
as well as Alexander Kechris, Andrew Marks and Todor Tsankov for many helpful comments.

\section{Preliminaries and examples}\label{Preliminaries}
\subsection{Descriptive set theory}
A \textbf{Polish space} is a second countable, completely metrizable topological space.
A \textbf{Borel equivalence relation} on a Polish space $X$ is an equivalence relation $E$
which is Borel as a subset of $X\times X$.
A Borel equivalence relation is \textbf{countable} (resp. \textbf{finite})
if every class is countable (resp. finite).
A countable Borel equivalence relation $E$ on $X$ is \textbf{smooth}
if there is a Borel function $f:X\to\R$ such that $xEx'$ if and only if $f(x) = f(x')$.
A Borel equivalence relation $E$ is \textbf{hyperfinite} (resp., \textbf{hypersmooth})
if $E = \bigcup_n E_n$,
where each $E_n\subset E_{n+1}$ (as a subset of $X\times X$)
and each $E_n$ is a finite (resp., smooth) Borel equivalence relation.
Given a Borel action of a countable group $\Gamma$ on a Polish space $X$,
we denote by $E^X_\Gamma$ the \textbf{orbit equivalence relation} of $\Gamma\car X$,
the Borel equivalence relation whose classes are the orbits of the action.
We will say that $\Gamma\car X$ is hyperfinite (resp., smooth, hypersmooth)
if its orbit equivalence relation $E^X_\Gamma$ is hyperfinite (resp., smooth, hypersmooth).

A \textbf{Polish group} is a topological group whose topology is Polish.
If $G$ is a Polish group and $H\le G$ is a closed subgroup,
then the quotient topology on the coset space $G/H$ is Polish
(see \cite[1.2.3]{BK96}).

\subsection{Countable subgroups of Polish groups}
Let $G$ be a Polish group and let $\Gamma\le G$ be a countable subgroup.
When clear from context,
we will abuse notation and identify
$G/\Gamma$ with the coset equivalence relation induced by $\Gamma\car G$
(technically,
$G/\Gamma$ is induced by the right action $G\caleft\Gamma$,
but this is isomorphic to the left action $\Gamma\car G$ via inversion).
For example,
we will say that $G/\Gamma$ is hyperfinite if $\Gamma\car G$ is hyperfinite.

Note that since the action $\Gamma\car G$ is free,
$G/\Gamma$ cannot be universal among countable Borel equivalence relations (see \cite[3.10]{Tho09}).

\begin{eg}
  We give below some examples of $G/\Gamma$ and the associated Borel complexity,
  for various $G$ and $\Gamma$:
  \begin{enumerate}
    \item $\R/\Z$ is smooth,
      since $\Z\le\R$ is a discrete subgroup
      (see \cite[7.2.1(iv)]{Kan08}).
    \item $\R/\Q$ is not smooth,
      since $\Q\le\R$ is a dense subgroup
      (see \cite[6.1.10]{Gao09}).
      Similarly,
      the commensurability relation $\R^+/\Q^+$ is not smooth.
      Note that both are hyperfinite,
      since they arise from Borel actions of countable abelian groups
      (see \cite[8.2]{GJ15}).
    \item Let $F_2\le\SO_3(\R)$ be a free subgroup on two generators.
      Then $\SO_3(\R)/F_2$ is not hyperfinite,
      since the free action $F_2\car\SO_3(\R)$ preserves the Haar measure
      (see \cite[7.4.8]{Gao09}).
    \item If $\Gamma$ is a countable group,
      then $\Inn(\Gamma)$ is a countable subgroup of $\Aut(\Gamma)$,
      which is a Polish group under the pointwise convergence topology,
      and we can consider the quotient $\Out(\Gamma) = \Aut(\Gamma)/\Inn(\Gamma)$.
      For example,
      when $\Gamma = S_\fin$ (the group of finitely supported permutations on $\N$),
      we have $\Out(S_\fin) \cong S_\infty/S_\fin$,
      which is hyperfinite and non-smooth.
  \end{enumerate}
\end{eg}
In the first, second and fourth examples,
$\Gamma$ is a normal subgroup of $G$.

\section{Proofs}
For any group $G$ and any subset $S\subset G$,
let $C_G(S)$ denote the centralizer of $S$ in $G$:
\[
  C_G(S) := \{g\in G:\forall s\in S(gs = sg)\}.
\]
Note that if $G$ is a topological group,
then $C_G(S)$ a closed subgroup of $G$.

\begin{prop}\label{OpenCentralizer}
  Let $G$ be a Baire group
  (i.e.,
  a topological group for which the Baire category theorem holds),
  and let $\Gamma$ be a finitely generated subgroup of $G$
  each of whose elements has countable conjugacy class in $G$.
  Then $C_G(\Gamma)$ is open in $G$.
\end{prop}

\begin{proof}
  Let $\Gamma = \ang{\gamma_0,\ldots,\gamma_n}$.
  Since each $\gamma_i$ has countable conjugacy class in $G$,
  we have $[G:C_G(\gamma_i)]\le\aleph_0$,
  so by the Baire category theorem,
  $C_G(\gamma_i)$ is nonmeager.
  Thus $C_G(\gamma_i)$ is an open subgroup of $G$,
  and thus $C_G(\Gamma)$ is also open,
  since $C_G(\Gamma) = \bigcap_{i<n} C_G(\gamma_i)$.
\end{proof}

Note that when $Z(\Gamma)$ is finite,
this immediately implies that $\Gamma$ is a discrete subgroup of $G$.

When $G$ is a compact group,
\Cref{OpenCentralizer} implies the following algebraic restriction on $\Gamma$:
\begin{thm}
  Let $G$ be a compact group and let $\Gamma$ be a countable normal subgroup of $G$.
  Then $\Gamma$ is locally virtually abelian,
  i.e. every finitely generated subgroup of $\Gamma$ is virtually abelian.
\end{thm}

\begin{proof}
  Let $\Delta$ be a finitely generated subgroup of $\Gamma$.
  Then by \Cref{OpenCentralizer}, $C_G(\Delta)$ is an open subgroup of $G$,
  so since $G$ is compact,
  the index of $C_G(\Delta)$ in $G$ is finite.
  Thus since $Z(\Delta) = \Delta\cap C_G(\Delta)$,
  the index of $Z(\Delta)$ in $\Delta$ is finite.
\end{proof}

We now prove the main theorem.
\begin{thm}\label{ThmHF}
  Let $G$ be a Polish group and let $\Gamma$ be a countable normal subgroup of $G$.
  Then $G/\Gamma$ is hyperfinite.
\end{thm}

\begin{proof}
  Let $\Gamma = (\gamma_k)_{k<\omega}$ and denote $\Gamma_k := \ang{\gamma_0,\ldots,\gamma_k}$.
  Let $C_k := C_G(\Gamma_k) = C_G(\gamma_0,\ldots,\gamma_k)$
  and let $Z_k := Z(\Gamma_k) = C_k\cap\Gamma_k$
  be the center of $\Gamma_k$.
  By \Cref{OpenCentralizer},
  $C_k$ is an open subgroup of $G$.
  
  Let $A := \ang{Z_k}_{k<\omega}$,
  the subgroup of $G$ generated by the $Z_k$ for all $k<\omega$.
  Then $A$ is an abelian subgroup of $G$,
  since each $Z_k$ is abelian,
  and since $Z_k$ commutes with $Z_l$ (pointwise) for any $k < l$.
  
  The principal fact we use about $A$ is the following:
  \begin{lem}
    $\Gamma\car G/\bar A$ is hyperfinite.
  \end{lem}
  
  \begin{proof}
    Since $\bar A$ is a closed subgroup of $G$,
    the coset space $G/\bar A$ is a standard Borel space,
    and thus $\Gamma\car G/\bar A$ induces a Borel equivalence relation.
    
    Every hypersmooth countable Borel equivalence relation is hyperfinite (see \cite[5.1]{DJK94}),
    so it suffices to show that $\Gamma\car G/\bar A$ is hypersmooth.
    Since $\Gamma$ is the increasing union of $(\Gamma_n)_n$,
    it suffices to show that $\Gamma_n\car G/\bar A$ is smooth.
    In fact,
    we will show that every orbit of $\Gamma_n\car G/\bar A$ is discrete,
    which implies smoothness
    (enumerate a basis,
    then for each orbit,
    find the first basic open set isolating an element of the orbit,
    and select that element).
    
    Fix $g\in G$,
    and fix $m$ large enough such that $g^{-1}\Gamma_n g\subset\Gamma_m$
    (for instance, by normality of $\Gamma$,
    take any $\Gamma_m$ containing $\{g^{-1}\gamma_i g\}_{i<n}$).
    We claim that the open neighbourhood $gC_m\bar A$ of $g\bar A$
    contains no other elements of the $\Gamma_n$-orbit of $g\bar A$.
    Suppose that $\gamma g\bar A\subset gC_m\bar A$ for some $\gamma\in\Gamma_n$,
    so that $g^{-1}\gamma g\in C_m\bar A$.
    Since $C_m$ is an open subgroup of $G$,
    it follows that $C_m A$ is closed
    (since its complement is a union of cosets of an open subgroup),
    so $C_m\bar A\subset \ol{C_m A} = C_m A$,
    and thus $g^{-1}\gamma g\in C_m A$.
    
    Write $g^{-1}\gamma g = ca$ for some $c\in C_m$ and $a\in A$.
    Now $A = \ang{Z_k}_{k<\omega}$,
    but since $A$ is abelian,
    we have $A = \ang{Z_k}_{k > m}\ang{Z_l}_{l\le m}$,
    and thus we can write $a = d\beta$,
    where $d\in\ang{Z_k}_{k > m}\subset C_m$
    and $\beta\in\ang{Z_l}_{l\le m}\subset\Gamma_m\cap A$.
    Now we have $g^{-1}\gamma g = cd\beta$,
    but since $g^{-1}\gamma g$ and $\beta$ are in $\Gamma_m$,
    we have $cd\in\Gamma_m$.
    Since $c$ and $d$ are in $C_m$ as well,
    we get $cd\in Z_m\subset A$.
    Now since $\beta$ is also in $A$,
    we have $g^{-1}\gamma g = cd\beta\in A\subset\bar A$,
    and thus $g\bar A = \gamma g\bar A$.
    \end{proof}
  
  We now use this lemma to show that $E^G_\Gamma$
  is induced by the action of a countable abelian group.
  This is sufficient since by a theorem of Gao and Jackson,
  every orbit equivalence relation of a countable abelian group is hyperfinite (\cite[8.2]{GJ15}).
  
  Since $\Gamma\car G/\bar A$ is hyperfinite,
  its orbit equivalence relation is generated
  by a Borel automorphism $T$ of $G/\bar A$
  (see \cite[5.1]{DJK94}).
  For each left $\bar A$-coset $C$,
  let $\gamma_{(C)}\in\Gamma$ be minimal such that $T(C) = \gamma_{(C)} C$,
  and let $U:G\to G$ be the Borel automorphism defined by $U(g) = \gamma_{(g\bar A)}g$
  (the inverse is defined by $g\mapsto(\gamma_{(T^{-1}(g\bar A))})^{-1}g$).
  This induces a Borel action $\Z\car G$,
  denoted $(n,g)\mapsto n\cdot g$,
  such that
  \begin{enumerate}
    \item $g\bar A$ and $h\bar A$ are in the same $\Gamma$-orbit
    iff for some $n$, $(n\cdot g)\bar A = h\bar A$,
    \item $E^G_\Z\subset E^G_\Gamma$,
    \item and $\Z\car G$ commutes with the right multiplication action $G\caleft(\bar A\cap\Gamma)$.
  \end{enumerate}
  So there is a Borel action of $\Z\times(\bar A\cap\Gamma)$ on $G$
  such that $E^G_{\Z\times(\bar A\cap\Gamma)}\subset E^G_\Gamma$.
  We claim that in fact,
  $E^G_{\Z\times(\bar A\cap\Gamma)} = E^G_\Gamma$.
  Suppose that $g,h\in G$ are in the same $\Gamma$-coset
  (note that we don't need to specify left/right since $\Gamma$ is normal).
  Then $g\bar A$ and $h\bar A$ are in the same $\Gamma$-orbit,
  so there is some $n\in\Z$ with $(n\cdot g)\bar A = h\bar A$.
  Since $E^G_\Z\subset E^G_\Gamma$,
  we have that $n\cdot g$ is in the same $\Gamma$-coset as $g$,
  and thus in the same $\Gamma$-coset as $h$.
  Since $n\cdot g$ and $h$ are in the same left $\bar A$-coset,
  and also in the same (left) $\Gamma$-coset,
  they are in the same left $\bar A\cap\Gamma$-coset.
  Thus $E^G_{\Z\times(\bar A\cap\Gamma)} = E^G_\Gamma$.
  Since $A$ is abelian,
  $\bar A$ is also abelian,
  and thus $\Z\times(\bar A\cap\Gamma)$ is a countable abelian group.
  So $E^G_\Gamma$ is generated by the action
  of a countable abelian group,
  and is therefore hyperfinite.
\end{proof}

We can extend this result to a slightly more general class of subgroups:
\begin{cor}
  Let $G$ be a Polish group and let $\Gamma\le G$ be a countable subgroup of $G$
  each of whose elements has countable conjugacy class in $G$.
  Then $G/\Gamma$ is hyperfinite.
\end{cor}

\begin{proof}
  Since every element of $\Gamma$ has countable conjugacy class in $G$,
  the subgroup $\Delta := \ang{g\Gamma g^{-1}}_{g\in G}$ is a countable normal subgroup of $G$,
  and thus by \Cref{ThmHF},
  $E^G_\Delta$ is hyperfinite.
  Since $E^G_\Gamma \subset E^G_\Delta$,
  we have that $G/\Gamma$ is also hyperfinite
  (since hyperfiniteness is closed under subequivalence relations).
\end{proof}

\begin{cor}
  Let $\Gamma$ be a countable group.
  Then $\Out(\Gamma)$ is hyperfinite.
\end{cor}

\begin{proof}
  This follows from \Cref{ThmHF},
  since $\Inn(\Gamma)\tl\Aut(\Gamma)$.
\end{proof}

We end with some open questions:
\begin{qn}
  Let $G$ be a Polish group and let $\Gamma$ be a countable subgroup.
  What are the possible Borel complexities of $G/\Gamma$?
  In particular,
  are they cofinal among orbit equivalence relations arising from free actions?
\end{qn}

\begin{defn}
  For a Polish group $G$,
  define the subgroup $Z_\omega(G)$ as follows:
  \[
    Z_\omega(G)
    := \{g\in G:\text{$g$ has countable conjugacy class}\}.
  \]
\end{defn}

In general,
$Z_\omega(G)$ is a characteristic subgroup of $G$,
analogous to the FC-center,
and $Z_\omega(G)$ is $\mathbf\Pi^1_1$ by Mazurkiewicz-Sierpi\'nski (see \cite[29.19]{Kec95}).

\begin{qn}
  Is there a Polish group $G$ such that $Z_\omega(G)$ is $\mathbf\Pi^1_1$-complete?
\end{qn}

\bibliography{quotient_hyperfinite}
\bibliographystyle{alpha}

\end{document}